\documentclass{amsart}

\usepackage{amsmath}
\usepackage{amssymb}
\usepackage{amsthm}
\usepackage{ascmac}
\usepackage{cases}
\usepackage{multicol}
\usepackage[all]{xy}
\usepackage{nccmath}
\usepackage{graphicx}
\usepackage{float}
\usepackage{color}

\newtheorem{theorem}{Theorem}[section] 
\newtheorem{proposition}[theorem]{Proposition} 
\newtheorem{lemma}[theorem]{Lemma}
\newtheorem{corollary}[theorem]{Corollary}

\newtheorem{definition}[theorem]{Definition}

\newcommand{\id}{\mathrm{id}}
\newcommand{\os}{\mathbin{\overline{*}}}
\newcommand{\us}{\mathbin{\underline{*}}}
\newcommand{\type}{\operatorname{type}}
\newcommand{\flow}{\mathrm{Flow}}
\newcommand{\col}{\mathrm{Col}}

\newcommand{\Q}{\mathcal{Q}}
\newcommand{\pr}{\mathrm{pr}}
\newcommand{\MCQ}{\mathsf{MCQ}}
\newcommand{\MCB}{\mathsf{MCB}}

\begin{document}

\title[A relationship between MCQ/MCB colorings]
{A relationship between multiple conjugation quandle/biquandle colorings}

\author{Tomo Murao}

\address
{Institute of Mathematics, University of Tsukuba,1-1-1 Tennoudai, Tsukuba, Ibaraki 305-8571, Japan.}
\email{t-murao@math.tsukuba.ac.jp}



\date{}

\begin{abstract}
We define a functor $\Q$ from the category of multiple conjugation biquandles 
to that of multiple conjugation quandles.
We show that 
for any multiple conjugation biquandle $X$, 
there is a one-to-one correspondence 
between the set of $X$-colorings  
and that of $\Q(X)$-colorings diagrammatically 
for any handlebody-link and spatial trivalent graph.
In particular, 
we prove that 
the set of $G$-family of Alexander biquandles colorings 
is isomorphic to 
that of $G$-family of Alexander quandles colorings 
as modules.
\end{abstract}

\maketitle

\section{Introduction}

A quandle \cite{Joyce82, Matveev82} is an algebraic system whose axioms are derived from the
Reidemeister moves on oriented link diagrams, 
and a biquandle \cite{FJK04,FRS95,KR03} is a generalization of a quandle.
The two algebraic systems yield many invariants 
for not only classical links but also surface links, virtual links and so on.
In particular, some invariants obtained from biquandles are stronger than those obtained from quandles for virtual links \cite{Kauffman99}.

On the other hand, 
as a corollary of \cite{Soloviev00}, 
it follows that 
there is a one-to-one correspondence between 
the set of biquandle colorings and that of quandle colorings 
for any classical links, 
where in the proof of the statement, 
any classical link need to be represented by a closed braid diagram.

Recently, 
Ishikawa \cite{Ishikawapre} constructed a left adjoint functor $\mathcal{B}$ of a functor $\mathcal{Q}$ from the category of biquandles to that of quandles 
which is defined in \cite{Ashihara14}.
By using $\mathcal{B}$, 
he proved that 
we can reconstruct a fundamental biquandle from a fundamental quandle, 
and there is a one-to-one correspondence between 
the set of biquandle colorings and that of quandle colorings 
for any classical and surface links, 
where in the statement, 
we can choose any diagram for classical and surface links.
Here we note that 
any left adjoint functor of the functor $\Q$ 
from the category of multiple conjugation biquandles to that of multiple conjugation quandles, 
which we will define in section 4 in this paper, 
has not been defined yet.
Furthermore, 
Ishikawa and Tanaka \cite{ITpre} explained the one-to-one correspondence proved in \cite{Ishikawapre} 
diagrammatically and concretely for classical and surface links.

Ishii, Iwakiri, Jang and Oshiro \cite{IIJO13} introduced a $G$-family of quandles 
to define colorings and invariants for handlebody-links and spatial trivalent graphs.
A multiple conjugation quandle (MCQ) is introduced in \cite{Ishii15} 
as a universal symmetric quandle with a partial multiplication 
to define coloring invariants for handlebody-links, 
where a partial multiplication is an operation used at trivalent vertices.
A $G$-family of biquandles \cite{IN17} and a multiple conjugation biquandle (MCB) \cite{IIKKMOpre,IN17} 
are biquandle versions of those algebraic systems.
However, 
although MCB colorings require more calculation than 
MCQ colorings in general, 
it has not been known whether 
an invariant  obtained from MCB colorings is more effective than 
one obtained from MCQ colorings.
In this paper, 
we partially extend the result in \cite{Ishikawapre,ITpre} to MCQ and MCB colorings 
for handlebody-links and spatial trivalent graphs.
Concretely, 
we show that 
for any handlebody-links and spatial trivalent graphs, 
there is a one-to-one correspondence between 
the set of MCB colorings and that of MCQ colorings diagrammatically (Theorem \ref{1:1 correspondence}).
We also show that 
the set of $G$-family of Aleander biquandles colorings 
is isomorphic to 
that of $G$-family of Alexander quandles colorings 
as modules (Corollary \ref{isomorphic}).

This paper is organized into five sections.
In Section 2, 
we review the definitions of 
a handlebody-link, a spatial trivalent graph 
and the Reidemeister moves of their diagrams.
In Section 3,  
we recall basic notions and facts about quandles and biquandles.
In Section 4, 
we review the definitions of a multiple conjugation quandle (MCQ) and a multiple conjugation biquandle (MCB) 
and define a functor $\Q$ from the category of MCBs to that of MCQs.
Moreover, we introduce colorings for handlebody-links and spatial trivalent graphs 
by using an MCQ and an MCB 
and show that 
for any MCB $X$, 
there is a one-to-one correspondence 
between the set of $X$-colorings 
and that of $\Q(X)$-colorings diagrammatically 
for any handlebody-link and spatial trivalent graph.
In Section 5, 
we introduce colorings for handlebody-links and spatial trivalent graphs by using a $G$-family of quandles and a $G$-family of biquandles.
We discuss the similar correspondence between the sets of colorings by using them.

\section{Preliminaries}
A \emph{handlebody-link} is the disjoint union of handlebodies embedded in the 3-sphere $S^3$ \cite{Ishii08}.
A \emph{handlebody-knot} is a handlebody-link with one component.
In this paper, 
we assume that every component of a handlebody-link 
is of genus at least $1$.
An \emph{$S^1$-orientation} of a handlebody-link is an orientation 
of all genus 1 components of the handlebody-link, 
where an orientation of a solid torus is an orientation of its core $S^1$.
Two $S^1$-oriented handlebody-links are \emph{equivalent} 
if there exists an orientation-preserving self-homeomorphism of $S^3$ 
sending one to the other 
preserving the $S^1$-orientation.

A \emph{spatial trivalent graph} is 
a finite trivalent graph embedded in $S^3$.
In this paper, 
a trivalent graph may have a circle component, 
which has no vertices.
A \emph{Y-orientation} of a spatial trivalent graph is 
a direction of all edges of the graph 
satisfying that every vertex of the graph is both 
the initial vertex of a directed edge 
and the terminal vertex of a directed edge (Figure \ref{Y-orientation}).
For a Y-oriented spatial trivalent graph $K$ 
and an $S^1$-oriented handlebody-link $H$, 
we say that 
$K$ \emph{represents} $H$ 
if $H$ is a regular neighborhood of $K$ 
and the $S^1$-orientation of $H$ agrees with the Y-orientation of $K$.
Then any $S^1$-oriented handlebody-link can be represented by 
some Y-oriented spatial trivalent graph.
We define a \emph{diagram} of an $S^1$-oriented handlebody-link 
by a diagram of a Y-oriented spatial trivalent graph 
representing the handlebody-link.
Then the following theorem holds.

\begin{figure}[htb]
\begin{center}
\includegraphics[width=45mm]{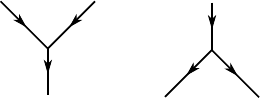}
\end{center}
\caption{Y-orientations.}\label{Y-orientation}
\end{figure}

\begin{theorem}[\cite{Ishii15-2}]\label{Reidemeister moves}
For a diagram $D_i$ of an $S^1$-oriented handlebody-link $H_i$ $(i=1,2)$, 
$H_1$ and $H_2$ are equivalent 
if and only if 
$D_1$ and $D_2$ are related 
by a finite sequence of R1--R6 moves 
depicted in Figure \ref{Reidemeister move} 
preserving Y-orientations, called the Reidemeister moves.
\end{theorem}

\begin{figure}[htb]
\begin{center}
\includegraphics[width=125mm]{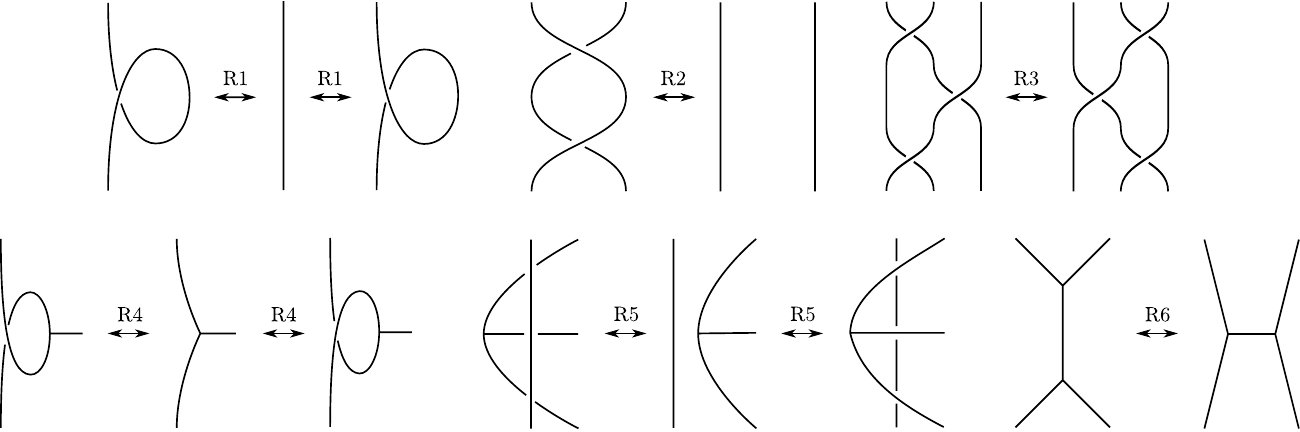}
\caption{The Reidemeister moves for handlebody-links.}\label{Reidemeister move}
\end{center}
\end{figure}

Here we note that 
the R1--R5 moves in Figure \ref{Reidemeister move} are 
the Reidemeister moves for spatial trivalent graphs \cite{Kauffman89,Yetter89}.
Hence we can regard handlebody-links as  a quotient structure of spatial trivalent graphs.

In this paper, 
we denote by $\mathcal{A}(D)$ and $\mathcal{SA}(D)$ 
the set of all arcs of $D$ and that of all semi-arcs of $D$ respectively, 
where a semi-arc is a piece of a curve 
each of whose endpoints is a crossing or a vertex.
An orientation of a (semi-)arc of $D$ 
is also represented by the normal orientation 
obtained by rotating the usual orientation counterclockwise 
by $\pi/2$ on the diagram.
For any $m \in \mathbb{Z}_{\geq 0}$,
we put 
$\mathbb{Z}_m:=\mathbb{Z}/m\mathbb{Z}$.
For an $S^1$-oriented handlebody-link $H$, 
the \emph{reverse} of $H$, denoted $-H$, 
is obtained by reversing the orientations of all genus 1 components, 
and 
the \emph{reflection} of $H$, denoted $H^*$, is the image of $H$ under 
an orientation-reversing self-homeomorphism of $S^3$.
A \emph{split handlebody-link} is a handlebody-link whose exterior is reducible.
For any handlebody-links $H_1$ and $H_2$, 
we denote by $H_1 \sqcup H_2$ the split handlebody-link $H$ 
such that there exists a 2-sphere in $S^3-H$ separating $S^3$ into two 3-balls, 
each of which contains only $H_1$ and $H_2$ respectively.
In this paper, we often omit brackets.
When we omit brackets, 
we apply binary operations from left on expressions, 
except for group operations, 
which we always apply first.
For example, 
we write $a *_1 b *_2 cd *_3 (e *_4 f *_5 g)$ 
for $((a *_1 b) *_2 (cd)) *_3 ((e *_4 f) *_5 g)$ simply, 
where each $*_i$ is a binary operation, 
and $c$ and $d$ are elements of the same group.

\section{Quandles and biquandles}

We recall the definitions of a quandle and a biquandle.

\begin{definition}[\cite{Joyce82, Matveev82}]
A \emph{quandle} is a non-empty set $X$ with a binary operation $* : X \times X \to X$ 
satisfying the following axioms.
\begin{itemize}
\item
For any $x \in X$, 
$x*x=x$.
\item
For any $y \in X$, 
the map $S_y : X \to X$ defined by $S_y(x)=x*y$ is a bijection.
\item
For any $x,y,z \in X$, 
$(x*y)*z=(x*z)*(y*z)$.
\end{itemize}
\end{definition}

We define the \emph{type} of a quandle $X$ by 
\begin{align*}
\type X=\min \{ n \in \mathbb{Z}_{>0} \mid a*^n b=a ~(\text{for any~} a,b \in X) \},
\end{align*}
where we set $a*^i b :=S_b^i(a)$ and $\min \emptyset :=\infty$ 
for any $i \in \mathbb{Z}$, $a,b \in X$ and the empty set $\emptyset$.
Any finite quandle is of finite type.

Let $X$ be an $R[t^{\pm1}]$-module, 
where $R$ is a commutative ring.
For any $a,b \in X$, we define $a * b=ta+(1-t)b$.
Then $X$ is a quandle, called an \emph{Alexander quandle}.

\begin{definition}[\cite{FJK04,FRS95,KR03}]
A \emph{biquandle} is a non-empty set $X$ 
with binary operations $\os, \us : X \times X \to X$ 
satisfying the following axioms.
\begin{itemize}
\item
For any $x \in X$, 
$x \us x=x \os x$.
\item
For any $y \in X$, 
the map $\underline{S}_y : X \to X$ 
defined by $\underline{S}_y(x)=x \us y$ is a bijection.
\item[]
For any $y \in X$, 
the map $\overline{S}_y : X \to X$ 
defined by $\overline{S}_y(x)=x \os y$ is a bijection.
\item[]
The map $S : X \times X \to X \times X$ 
defined by $S(x,y)=(y \os x,x \us y)$ is a bijection.
\item
For any $x,y,z \in X$, 
\begin{align*}
(x \us y) \us (z \us y)=(x \us z) \us (y \os z),\\
(x \us y) \os (z \us y)=(x \os z) \us (y \os z),\\
(x \os y) \os (z \os y)=(x \os z) \os (y \us z).
\end{align*}
\end{itemize}
\end{definition}

We note that 
$(X,*)$ is a quandle 
if and only if 
$(X,*,\os)$ is a biquandle 
with $x \os y=x$.
Let  $(X,\us,\os)$ be a biquandle.
For any $i \in \mathbb{Z}$ and $a,b \in X$, 
we define $a \us^i b := \underline{S}^i_b(a)$ and $a \os^i b := \overline{S}^i_b(a)$.
Then we define two families of binary operations 
$\us^{[n]}, \os^{[n]} : X \times X \to X (n \in \mathbb{Z})$ 
by the equalities: 
\begin{align*}
& a \us^{[0]}b=a,~~ a \us^{[1]}b=a \us b,~~ a \us^{[i+j]}b=(a\us^{[i]}b)\us^{[j]}(b\us^{[i]}b),\\
& a \os^{[0]}b=a,~~ a \os^{[1]}b=a \os b,~~ a \os^{[i+j]}b=(a\os^{[i]}b)\os^{[j]}(b\os^{[i]}b)
\end{align*}
for any $i,j \in \mathbb{Z}$ \cite{IIKKMOpre,IN17}.
Since $a=a \us^{[0]}b=(a\us^{[-1]}b)\us^{[1]}(b\us^{[-1]}b)=(a\us^{[-1]}b)\us(b\us^{[-1]}b)$, 
we have $a \us^{[-1]}b=a\us^{-1}(b\us^{[-1]}b)$ 
and $(b\us^{[-1]}b)\us(b\us^{[-1]}b)=b$, 
where we note that 
$b \us^{[-1]} b$ is the unique element satisfying $(b \us^{[-1]} b) \us (b \us^{[-1]} b)=b$ \cite{IIKKMOpre}.

We define the \emph{type} of a biquandle $X$ by 
\begin{align*}
\type X=\min \{ n \in \mathbb{Z}_{>0} \mid a\us^{[n]}b=a=a\os^{[n]}b ~(\text{for any~} a,b \in X) \},
\end{align*}
where we remind that $\min \emptyset=\infty$ for the empty set $\emptyset$.
Any finite biquandle is of finite type \cite{IN17}.

Let $X$ be an $R[s^{\pm1},t^{\pm1}]$-module, 
where $R$ is a commutative ring.
For any $a,b \in X$, we define $a \us b=ta+(s-t)b$ and $a \os b=sa$.
Then $X$ is a biquandle, called an \emph{Alexander biquandle}.
Any Alexander biquandle with $s=1$ 
coincides with an Alexander quandle.
For an Alexander biquandle $X$, 
we have $a\us^{[n]}b=t^na+(s^n-t^n)b$ 
and $a\os^{[n]}b=s^na$ 
for any $a,b \in X$.

For any biquandle $(X,\us,\os)$, 
we have a quandle $(X,*)$, denoted by $\Q(X)$, 
by defining $x*y=x \us y \os^{-1}y$ for any $x,y \in X$ \cite{Ashihara14}.
There is a one-to-one correspondence between 
the set of $X$-colorings and that of $\Q(X)$-colorings 
for any classical link \cite{Soloviev00} and surface link \cite{Ishikawapre,ITpre}.

For any Alexander biquandle $X$, 
which is an $R[s^{\pm1},t^{\pm1}]$-module for some commutative ring $R$, 
$\Q(X)$ is the Alexander quandle, which is the $R[(s^{-1}t)^{\pm1}]$-module.
That is, for any $x,y \in \Q(X)$, 
it follows that 
$x*y=s^{-1}tx+(1-s^{-1}t)y$.

\begin{proposition}\label{type}
For any Alexander biquandle $X$ which is of finite type, 
$\type X$ is divisible by $\type \Q(X)$.
\end{proposition}

\begin{proof}
Put $m=\type X$ and $m'=\type \Q(X)$.
Then it follows that 
$x \us^{[m]}y=t^mx+(s^m-t^m)y=x$ and $x \os^{[m]} y=s^m x=x$ 
for any $x,y \in X$.
Hence we have $s^m=t^m=1$, that is, $x *^m y=s^{-m}t^mx+(1-s^{-m}t^m)y=x$ for any $x,y \in \Q(X)$.
Therefore we have $m' \leq m$.
We assume that $m=m'l_1+l_2$ for some $l_1,l_2 \in \mathbb{Z}_{\geq 0}$ such that $0 < l_2 <m'$.
Then we have $x *^m y=s^{-l_2}t^{l_2}x+(1-s^{-l_2}t^{l_2})y=x$, 
which contradicts to $m'=\type \Q (X)$.
Therefore we obtain $m=m'l_1$ for some $l_1 \in \mathbb{Z}_{\geq 0}$.
\end{proof}

Here we see two examples.
Let $X$ be the Alexander biquandle $\mathbb{Z}[s^{\pm 1},t^{\pm 1}]/(s-t)$.
Then we have $\type X=\infty$ and $\type \Q(X)=1$.
Next, 
let $X$ be the Alexander biquandle $\mathbb{Z}[s^{\pm 1},t^{\pm 1}]/(s+t,t^4-1)$.
Then we have $\type X=4$ and $\type \Q(X)=2$.

\section{A relationship between MCQ/MCB colorings}

In this section, 
we recall the definitions of a multiple conjugation quandle (MCQ) and a multiple conjugation biquandle (MCB) 
and define a functor $\Q$ from the category of MCBs to that of MCQs.
We prove that 
for any MCB $X$, 
there is a one-to-one correspondence between 
the set of $X$-colorings and that of $\Q(X)$-colorings 
for any $S^1$-oriented handlebody-link.

Firstly, we review the definition of a multiple conjugation quandle (MCQ).

\begin{definition}[\cite{Ishii15}]
A \emph{multiple conjugation quandle (MCQ)} $X$ 
is the disjoint union of groups $G_{\lambda} (\lambda \in \Lambda)$ 
with a binary operation $* : X \times X \to X$ 
satisfying the following axioms.
\begin{itemize}
\item
For any $a,b \in G_\lambda$, 
$a*b = b^{-1}ab$.
\item
For any $x \in X$ and $a,b \in G_\lambda$, 
$x*e_\lambda = x$ and $x*ab=(x*a)*b$, 
where $e_\lambda$ is the identity of $G_\lambda$.
\item
For any $x,y,z \in X$, 
$(x*y)*z=(x*z)*(y*z)$.
\item
For any $x \in X$ and $a,b \in G_\lambda$, 
$ab*x=(a*x)(b*x)$, 
where $a*x, b*x \in G_\mu$ for some $\mu \in \Lambda$.
\end{itemize}
\end{definition}

We remark that an MCQ itself is a quandle.
Let $X=\bigsqcup_{\lambda \in \Lambda}G_\lambda$ and $Y=\bigsqcup_{\mu \in M}G_\mu$ be MCQs.
An \emph{MCQ homomorphism} $\phi : X \to Y$ is a map 
from $X$ to $Y$ satisfying $\phi(x*y)=\phi(x)*\phi(y)$ for any $x,y \in X$ 
and $\phi(ab)=\phi(a)\phi(b)$ for any $\lambda \in \Lambda$ and $a,b \in G_\lambda$.
We call a bijective MCQ homomorphism an \emph{MCQ isomorphism}.
$X$ and $Y$ are \emph{isomorphic} 
if there exists an MCQ isomorphism from $X$ to $Y$.
There is a category $\MCQ$ of MCQs, whose objects are MCQs and whose morphisms are MCQ homomorphisms.

Next, we review the definition of a multiple conjugation biquandle (MCB).
Let $X$ be the disjoint union of groups $G_{\lambda} (\lambda \in \Lambda)$.
We denote by $G_a$ the group $G_\lambda$ containing $a \in X$.
We also denote by $e_\lambda$ the identity of $G_\lambda$.
Then the identity of $G_a$ is denoted by $e_a$ for any $a \in X$.

\begin{definition}[\cite{IIKKMOpre,IN17}]
A \emph{multiple conjugation biquandle (MCB)} $X$ 
is the disjoint union of groups $G_{\lambda} (\lambda \in \Lambda)$ 
with binary operations $\us, \os : X \times X \to X$ 
satisfying the following axioms.
\begin{itemize}
\item
For any $x,y,z \in X$, 
\begin{align*}
(x \us y) \us (z \us y)=(x \us z) \us (y \os z),\\
(x \us y) \os (z \us y)=(x \os z) \us (y \os z),\\
(x \os y) \os (z \os y)=(x \os z) \os (y \us z).
\end{align*}
\item
For any $a,x \in X$, 
$\us x : G_a \to G_{a \us x}$ and $\os x : G_a \to G_{a \os x}$ 
are group homomorphisms.
\item
For any $x \in X$ and $a,b \in G_\lambda$, 
\begin{align*}
& x \us ab =(x \us a) \us (b \os a), &x \us e_\lambda = x,\\
& x \os ab =(x \os a) \os (b \os a), &x \os e_\lambda = x,\\
& a^{-1}b \os a= ba^{-1} \us a.
\end{align*}
\end{itemize}
\end{definition}

We remark that an MCB itself is a biquandle.
Let $X=\bigsqcup_{\lambda \in \Lambda}G_\lambda$ and $Y=\bigsqcup_{\mu \in M}G_\mu$ be MCBs.
An \emph{MCB homomorphism} $\phi : X \to Y$ is a map 
from $X$ to $Y$ satisfying $\phi(x\us y)=\phi(x)\us \phi(y)$ and $\phi(x\os y)=\phi(x)\os \phi(y)$ for any $x,y \in X$ 
and $\phi(ab)=\phi(a)\phi(b)$ for any $\lambda \in \Lambda$ and $a,b \in G_\lambda$.
We call a bijective MCB homomorphism an \emph{MCB isomorphism}.
$X$ and $Y$ are \emph{isomorphic} 
if there exists an MCB isomorphism from $X$ to $Y$.
There is a category $\MCB$ of MCBs, whose objects are MCBs and whose morphisms are MCB homomorphisms.

\begin{definition}
We define a functor $\Q$ from $\MCB$ to $\MCQ$ by 
$\Q((X,\us,\os))=(X,*)$ with $x*y=x \us y \os^{-1} y$ for any MCB $(X,\us,\os)$ 
and $\Q(\phi)=\phi$ for any MCB homomorphism $\phi$.
\end{definition}

In the following, 
we see that the functor $\Q$ is well-defined.

\begin{proposition}\label{MCB to MCQ}
The functor $\Q : \MCB \to \MCQ$ is well-defined.
\end{proposition}

\begin{proof}
Let $X=\bigsqcup_{\lambda \in \Lambda}G_{\lambda}$ be an MCB.
At first, since $a^{-1}b \os a = ba^{-1} \us a$ for any $a,b \in G_{\lambda}$, 
we have $a*b=a \us b \os^{-1} b= b^{-1}ab$.
Second, for any $x \in X$ and $a,b \in G_\lambda$,
$x \os^{-1} ab=x \os^{-1} (a \us b) \os^{-1}b$ 
since 
\begin{align*}
x \os^{-1} (a \us b) \os^{-1} b \os ab 
&= x \os^{-1} (a \us b) \os^{-1} b \os b(a \us b \os^{-1} b)\\
&= (x \os^{-1} (a \us b) \os^{-1} b \os b) \os (a \us b \os^{-1} b \os b)\\
&= x \os^{-1} (a \us b) \os (a \us b)\\
&= x.
\end{align*}
Hence we have 
\begin{align*}
x*ab
&= x \us ab \os^{-1} ab\\
&= (x \us a) \us (b \os a) \os^{-1} (a \us b) \os^{-1} b\\
&= ((x \us a \os^{-1} a) \os a) \us (b \os a) \os^{-1} (a \us b) \os^{-1} b\\
&= ((x \us a \os^{-1} a) \us b) \os (a \us b) \os^{-1} (a \us b) \os^{-1} b\\
&= x \us a \os^{-1} a \us b \os^{-1} b\\
&= (x*a)*b.
\end{align*}
Furthermore, we can easily check that 
$x*e_\lambda= x$.
Third, for any $x,y,z \in X$, 
we obtain $(x*y)*z=(x*z)*(y*z)$ 
since $\Q(X)$ is a quandle \cite{Ashihara14}.
Finally, for any $x \in X$ and $a,b \in G_\lambda$,
\begin{align*}
ab*x
&=ab \us x \os^{-1} x \\
&= (a \us x \os^{-1} x)(b \us x \os^{-1} x) \\
&= (a*x)(b*x) 
\end{align*}
since 
$\us x : G_a \to G_{a \us x}$ and $\os x : G_a \to G_{a \os x}$ are group homomorphisms.
Therefore $\Q(X)$ is an MCQ.

On the other hand, 
for any MCB homomorphism $\phi : X \to Y$ and $x,y \in X$, 
we have 
\begin{align*}
\Q(\phi)(x*y) 
&= \phi(x*y)\\
&= \phi(x \us y \os^{-1}y)\\
&= \phi(x) \us \phi(y) \os^{-1} \phi(y)\\
&= \phi(x)*\phi(y)\\
&=\Q(\phi)(x)*\Q(\phi)(y).
\end{align*}
Hence $\Q(\phi)$ is an MCQ homomorphism from $\Q(X)$ to $\Q(Y)$.
Furthermore it is clear that 
$\mathcal{Q}(\id_X)=\id_{\mathcal{Q}(X)}$ 
and $\mathcal{Q}(\psi \circ \phi)=\mathcal{Q}(\psi) \circ \mathcal{Q}(\phi)$.
This completes the proof.
\end{proof}

Let $X=\bigsqcup_{\lambda \in \Lambda} G_\lambda$ be an MCQ (resp. MCB) 
and let $D$ be a diagram of an $S^1$-oriented handlebody-link $H$.
An \emph{$X$-coloring} of $D$ is a map 
$C : \mathcal{A}(D) \to X$ (resp. $\mathcal{SA}(D) \to X$) satisfying the conditions 
depicted in Figure \ref{MCQ-coloring} (resp. Figure \ref{MCB-coloring}) 
at each crossing and vertex.
We denote by $\col_X(D)$ 
the set of all $X$-colorings of $D$.

\begin{figure}[htb]
\begin{center}
\includegraphics[width=75mm]{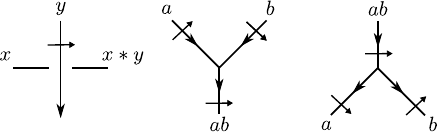}
\end{center}
\caption{An MCQ-coloring of $D$.}\label{MCQ-coloring}
\end{figure}

\begin{figure}[htb]
\begin{center}
\includegraphics[width=115mm]{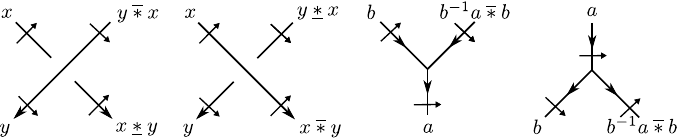}
\end{center}
\caption{An MCB-coloring of $D$.}\label{MCB-coloring}
\end{figure}

\begin{proposition}[\cite{Ishii15,IIKKMOpre,IN17}]\label{MCQ,MCB inv.}
Let $X=\bigsqcup_{\lambda \in \Lambda} G_\lambda$ be an MCQ or MCB 
and let $D$ be a diagram of an $S^1$-oriented handlebody-link $H$.
Let $D'$ be a diagram obtained by applying one of Reidemeister moves to the diagram $D$ once.
For an $X$-coloring $C$ of $D$, 
there is a unique $X$-coloring $C'$ of $D'$ which coincides with $C$ 
except near the point where the move is applied.
\end{proposition}

By this proposition, 
the cardinality of $X$-colorings of $D$ is an invariant of $H$.

Let $D$ and $D'$ be diagrams of $S^1$-oriented handlebody-links $H$ and $H'$ respectively.
In the following, 
we define diagrams $-D,D^v,D^h,D \sqcup D'$ and $W(D)$ (see Figure \ref{diagrams}).
We denote by $-D$ and $D^v$ the diagrams of $-H$ and $H^*$ obtained from $D$ 
by reversing the orientations of all (semi)-arcs 
and switching all crossings respectively.
We can regard that $D$ is depicted in an $xy$-plane.
Let $\iota$ be the involution $(x,y) \mapsto (-x,y)$.
Then we define the diagram $D^h$ of $H^*$ by $D^h= \iota (D)$.
We regard $\iota$ as the map from $\mathcal{A}(D)$ to $\mathcal{A}(D^h)$ 
(or $\mathcal{SA}(D)$ to $\mathcal{SA}(D^h)$).

An $S^1$-oriented handlebody-link diagram in $S^2$ is a \emph{split diagram} 
if there is a loop in the exterior of the diagram separating $S^2$ into two disks each containing part of it.
We denote by $D \sqcup D'$ the split diagram of $H \sqcup H'$ 
such that $D$ and $D'$ represent $H$ and $H'$ respectively.
We denote by $W(D)$ the diagram of the $S^1$-oriented handlebody-link $H \sqcup -H^*$ 
obtained from $D \sqcup -D^v$ 
by sliding $-D^v$ under $D$ and shifting it slightly to the normal orientations of all (semi-)arcs of $D$.

\begin{figure}[htb]
\begin{center}
\includegraphics[width=115mm]{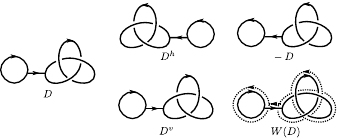}
\end{center}
\caption{Diagrams $D,-D,D^v,D^h$ and $W(D)$.}\label{diagrams}
\end{figure}

Let $X$ be an MCB.
We note here that $\mathcal{SA}(-D)=\mathcal{SA}(D)$.
For any $C \in \col_X(D)$, 
we define $C^* \in \col_X(-D^h)$ by $C^*=C \circ \iota$ 
as shown in Figure \ref{reflection coloring ex}, 
where each $x_i$ is an element of $X$.
We note that the $X$-coloring $C^*$ is shown in Figure \ref{reflection coloring} 
at each crossing and vertex.
We define $C \sqcup C^* \in \col_X(D \sqcup -D^h)$ by 
$(C \sqcup C^*)|_{\mathcal{SA}(D)}=C$ and $(C \sqcup C^*)|_{\mathcal{SA}(-D^h)}=C^*$.
We set $\col^\sqcup_X(D \sqcup -D^h) :=\{ C \sqcup C^* \mid C \in \col_X(D) \}$.
We denote by $\col^W_X(W(D))$ 
the set of $X$-colorings of $W(D)$ satisfying the conditions depicted in Figure \ref{double coloring} 
at each crossing and vertex.

\begin{figure}[htb]
\begin{center}
\includegraphics[width=80mm]{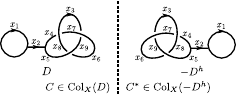}
\end{center}
\caption{Colorings $C$ and $C^*$.}\label{reflection coloring ex}
\end{figure}

\begin{figure}[htb]
\begin{center}
\includegraphics[width=125mm]{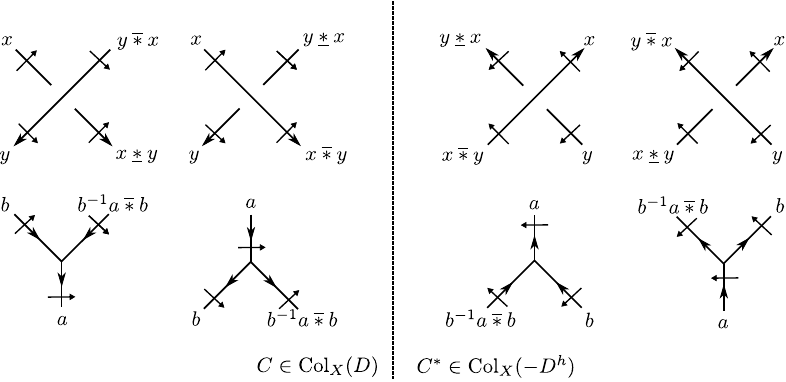}
\end{center}
\caption{The well-definedness of $C^* \in \col_X(-D^h)$.}\label{reflection coloring}
\end{figure}

\begin{figure}[htb]
\begin{center}
\includegraphics[width=120mm]{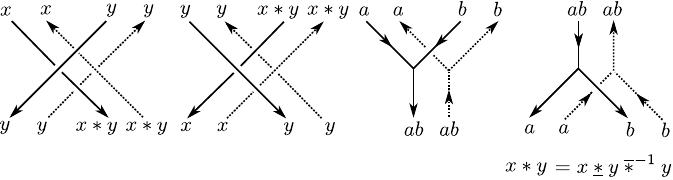}
\end{center}
\caption{The coloring conditions of $\col_X^W(W(D))$.}\label{double coloring}
\end{figure}

\begin{lemma}\label{braid coloring}
Let $X$ be an MCB. 
For the $X$-coloring depicted in Figure \ref{braid_coloring_proof3}, 
where $x_i$, $x_i'$, $y_i$, $y_i'$, $z_i$, $z_i'$, $w_i$ and $w_i'$ are elements of $X$ for any $i$, 
it follows that $(x_1,\ldots, x_l)=(x'_1,\ldots, x'_l)$ if and only if $(y_1,\ldots, y_l)=(y'_1,\ldots, y'_l)$.
\end{lemma}

\begin{figure}[htb]
\begin{center}
\includegraphics[width=100mm]{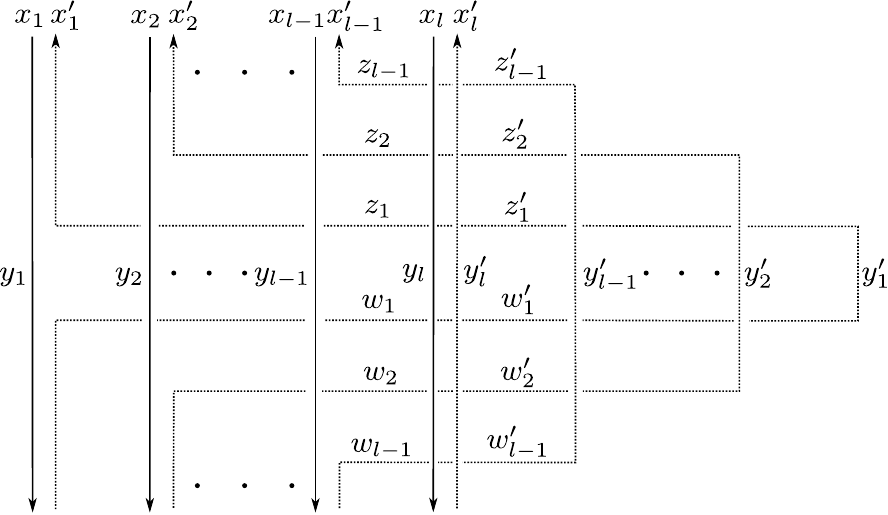}
\end{center}
\caption{}\label{braid_coloring_proof3}
\end{figure}

\begin{proof}
We give the proof by induction on $l$.
When $l=1$, the statement holds immediately.
Assume that the statement is proved for $l-1$.
Suppose that $(x_1,\ldots, x_l)=(x'_1,\ldots, x'_l)$.
Then we have $z_i=z_i'$, $y_l=y_l'$ and $w_i=w_i'$ for any $i=1, \ldots, l-1$ (see Figure \ref{braid_coloring_proof3}).
Hence we obtain the $X$-coloring depicted in Figure \ref{braid_coloring_proof4} 
from the $X$-coloring depicted in Figure \ref{braid_coloring_proof3}.
Therefore we have $(y_1,\ldots, y_{l-1})=(y'_1,\ldots, y'_{l-1})$ by the assumption.
Consequently, it follows that $(y_1,\ldots, y_l)=(y'_1,\ldots, y'_l)$.
In the same way, 
if we suppose that $(y_1,\ldots, y_l)=(y'_1,\ldots, y'_l)$, 
then it follows that $(x_1,\ldots, x_l)=(x'_1,\ldots, x'_l)$, 
where we also have 
$z_i=z_i'$ and $w_i=w_i'$ for any $i=1, \ldots, l-1$.
\end{proof}

\begin{figure}[htb]
\begin{center}
\includegraphics[width=100mm]{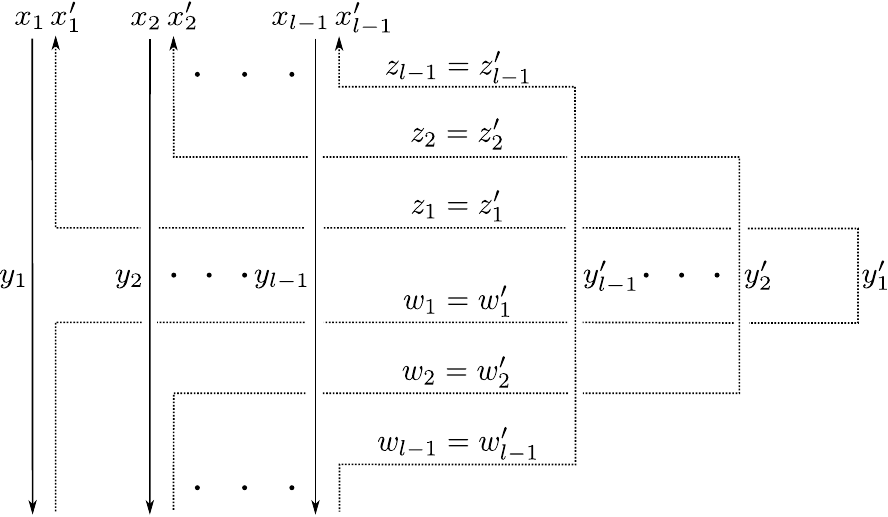}
\end{center}
\caption{}\label{braid_coloring_proof4}
\end{figure}

\begin{theorem}\label{1:1 correspondence}
Let $X$ be an MCB 
and let $D$ be a diagram of an $S^1$-oriented handlebody-link.
Then there is a one-to-one correspondence between 
$\col_X(D)$ and $\col_{\Q(X)}(D)$.
\end{theorem}

\begin{proof}
By \cite{Murao16}, 
any $S^1$-oriented handlebody-link can be represented by 
$\mathrm{cl}( \mathrm{bind}^{m_1,\ldots,m_s}_{n_1,\ldots,n_s}(b_0))$, 
where $b_0$ is a classical $l$-braid diagram 
and $m_i,n_i \in \mathbb{Z}_{>0}$, 
and we can deform it into $\mathrm{cl} (b)$, 
where $b$ is the trivalent braid diagram 
as shown in Figure \ref{closed braid}.
Then we may assume that $D$ has the resulting form $\mathrm{cl} (b)$.
Here we note that 
any MCQ(MCB)-coloring of $D$ 
is determined by the colors of all (semi-)arcs incident to the top endpoints of the trivalent braid diagram $b$.

First, for any $C_1 \in \col_{\Q(X)}(D)$, 
we denote by $\psi_1^X(C_1)$ the $X$-coloring of $W(D)$ depicted in Figure \ref{psi1}.
Then $\psi_1^X$ is a bijective map from $\col_{\Q(X)}(D)$ to $\col_X^W(W(D))$.
Second, 
we can deform $W(D)$ into $D \sqcup -D^h$ by Reidemeister moves as shown in Figure \ref{deformation}.
By Proposition \ref{MCQ,MCB inv.} and Lemma \ref{braid coloring}, 
we obtain a bijective map $\psi_2^X$ from $\col_X^W(W(D))$ to $\col_X^\sqcup(D \sqcup -D^h)$ 
as shown in Figure \ref{deformation}, 
where $x_i$ and $y_i$ are elements in $X$.
Finally, 
we define a map $\psi_3^X$ from $\col^\sqcup_X(D \sqcup -D^h)$ to $\col_X(D)$ 
by $\psi_3^X (C_3 \sqcup C_3^*)=C_3$, 
which is bijective obviously.
Therefore $\psi_3^X \circ \psi_2^X \circ \psi_1^X$ is a bijective map from $\col_{\Q(X)}(D)$ to $\col_X(D)$.
\end{proof}

\begin{figure}[htb]
\begin{center}
\includegraphics[width=112mm]{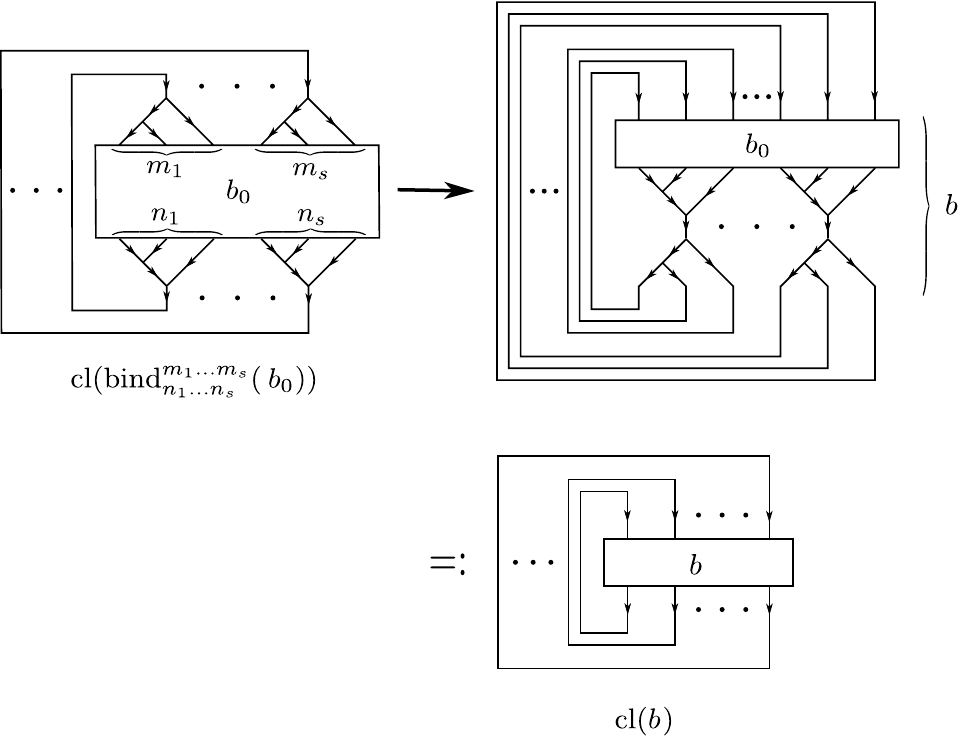}
\end{center}
\caption{A closed trivalent braid diagram.}\label{closed braid}
\end{figure}

\begin{figure}[htb]
\begin{center}
\includegraphics[width=125mm]{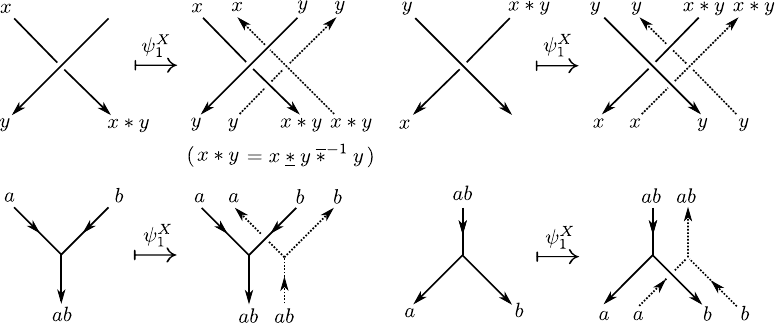}
\end{center}
\caption{The map $\psi_1^X : \col_{\Q(X)}(D) \to \col_X^W(W(D))$.}\label{psi1}
\end{figure}

\begin{figure}[htb]
\begin{center}
\includegraphics[width=132mm]{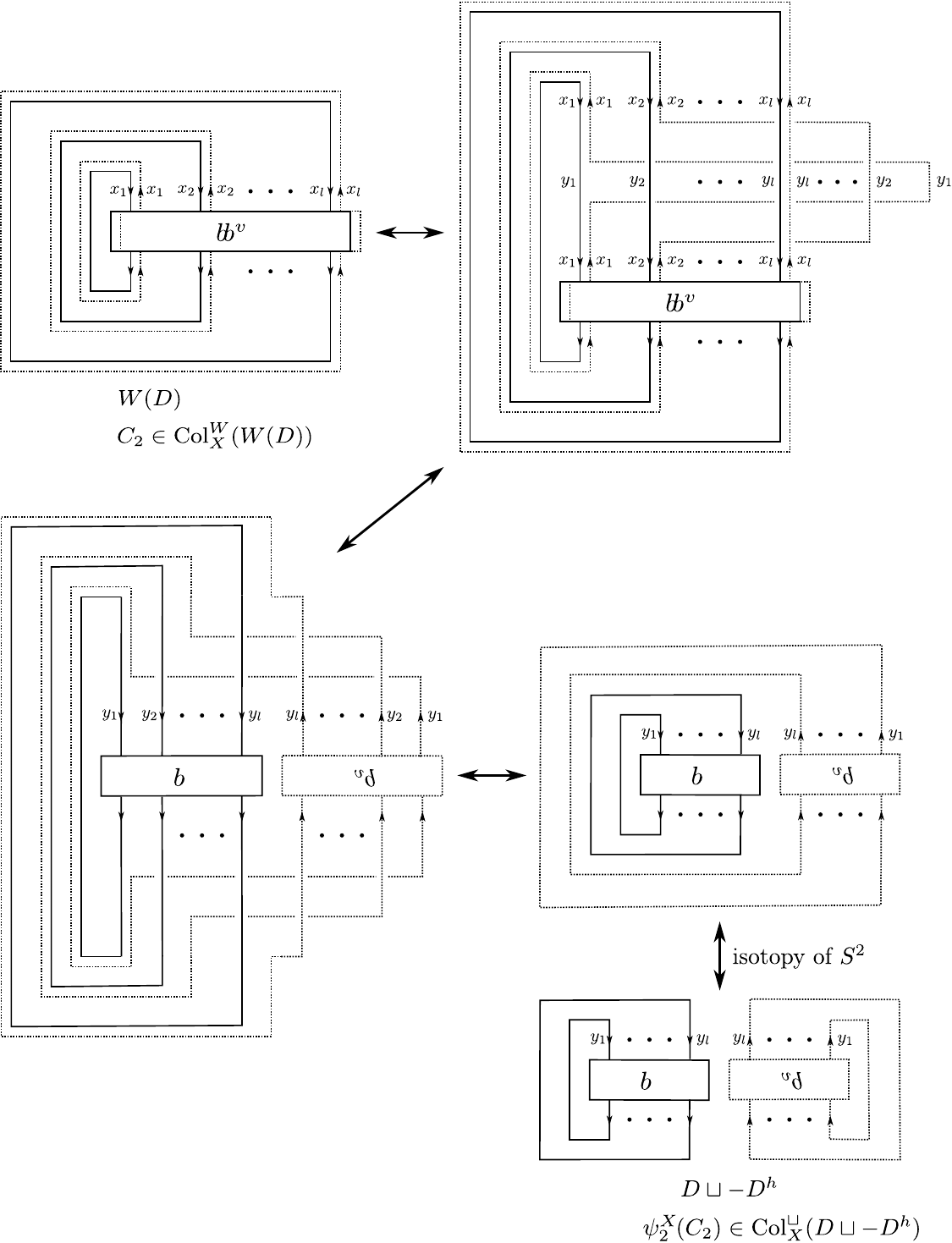}
\end{center}
\caption{The deformation from $W(D)$ to $D \sqcup -D^h$.}\label{deformation}
\end{figure}

\section{A relationship between $G$-family of quandles/biquandles colorings}

A $G$-family of quandles (resp. biquandles), which are algebraic systems 
whose axioms are motivated from handlebody-knot theory, 
yields an MCQ (resp. MCB).
In this section, 
we recall the definitions of a $G$-family of quandles and biquandles 
and define a map from the set of $G$-families of biquandles to that of $G$-families of quandles.
We prove that 
there is the similar correspondence 
between $G$-family of quandles and biquandles colorings 
to the one between MCQ and MCB colorings.

\begin{definition}[\cite{IIJO13}]
Let $G$ be a group with identity element $e$.
A \emph{$G$-family of quandles} is a non-empty set $X$ 
with a family of binary operations $*^g:X \times X \to X ~(g \in G)$ 
satisfying the following axioms.
\begin{itemize}
\item
For any $x \in X$ and $g \in G$, 
$x*^gx=x.$
\item
For any $x,y \in X$ and $g,h \in G$, 
$x*^{gh}y=(x*^gy)*^hy$ and $x*^ey=x$.
\item\
For any $x,y,z \in X$ and $g,h \in G$, 
$(x*^gy)*^hz=(x*^hz)*^{h^{-1}gh}(y*^hz)$.
\end{itemize}
\end{definition}

Let $R$ be a ring and let $G$ be a group with identity element $e$.
Let $X$ be a right $R[G]$-module, where $R[G]$ is the group ring of $G$ over $R$.
Then $(X,\{ *^g \}_{g \in G})$ is a $G$-family of quandles, called a \emph{$G$-family of Alexander quandles}, 
with $x*^gy=xg+y(e-g)$ \cite{IIJO13}.
Let $(X,*)$ be a quandle and let $m$ be the type of $X$.
Then $(X,\{*^i\}_{i \in \mathbb{Z}_{km}})$ is a $\mathbb{Z}_{km}$-family of quandles 
for any $k \in \mathbb{Z}_{\geq 0}$ \cite{IIJO13}.
In particular, 
when $X$ is an Alexander quandle, 
$(X,\{*^i\}_{i \in \mathbb{Z}_{km}})$ is called a \emph{$\mathbb{Z}_{km}$-family of Alexander quandles}.

Let $(X,\{ *^g \}_{g \in G})$ be a $G$-family of quandles.
Then $X \times G =\bigsqcup_{x \in X} \{ x \} \times G$ is an MCQ with 
\begin{align*}
(x,g)*(y,h):=(x*^hy,h^{-1}gh),
\hspace{10mm}
(x,g)(x,h):=(x,gh)
\end{align*}
for any $x,y \in X$ and $g,h \in G$ \cite{Ishii15}.
We call it the \emph{associated MCQ} 
of  $(X,\{ *^g \}_{g \in G})$.

\begin{definition}[\cite{IIKKMOpre,IN17}]
Let $G$ be a group with identity element $e$.
A \emph{$G$-family of biquandles} is a non-empty set $X$ 
with two families of binary operations $\us^g, \os^g :X \times X \to X ~(g \in G)$ 
satisfying the following axioms.
\begin{itemize}
\item
For any $x \in X$ and $g \in G$, 
\begin{align*}
x\us^gx=x\os^gx.
\end{align*}
\item
For any $x,y \in X$ and $g,h \in G$, 
\begin{align*}
& x\us^{gh}y=(x\us^gy)\us^h(y\us^gy),~~ x\us^ey=x,\\
& x\os^{gh}y=(x\os^gy)\os^h(y\os^gy),~~ x\os^ey=x.
\end{align*}
\item
For any $x,y,z \in X$ and $g,h \in G$, 
\begin{align*}
(x\us^g y)\us^h(z\os^gy)=(x\us^hz)\us^{h^{-1}gh}(y\us^hz),\\
(x\os^g y)\us^h(z\os^gy)=(x\us^hz)\os^{h^{-1}gh}(y\us^hz),\\
(x\os^g y)\os^h(z\os^gy)=(x\os^hz)\os^{h^{-1}gh}(y\us^hz).
\end{align*}
\end{itemize}
\end{definition}

Let $R$ be a ring, $G$ be a group with identity element $e$ 
and let $f : G \to Z(G)$ be a homomorphism, where $Z(G)$ is the center of $G$.
Let $X$ be a right $R[G]$-module, where $R[G]$ is the group ring of $G$ over $R$.
Then $(X, \{\us^g \}_{g \in G}, \{\os^g\}_{g \in G})$ is a $G$-family of biquandles, called a \emph{$G$-family of Alexander biquandles}, 
with $x \us^g y=xg+y(f(g)-g)$ and $x \os^g y=x f(g)$ \cite{IIKKMOpre}.
Let $(X,\us,\os)$ be a biquandle and let $m$ be the type of $X$.
Then $(X,\{\us^i\}_{i \in \mathbb{Z}_{km}}, \{\os^i\}_{i \in \mathbb{Z}_{km}})$ is a $\mathbb{Z}_{km}$-family of biquandles 
for any $k \in \mathbb{Z}_{\geq 0}$ \cite{IN17}.
In particular, 
when $X$ is an Alexander biquandle, 
$(X,\{\us^i\}_{i \in \mathbb{Z}_{km}}, \{\os^i\}_{i \in \mathbb{Z}_{km}})$ is called a \emph{$\mathbb{Z}_{km}$-family of Alexander biquandles}.

Let $(X,\{ \us^g \}_{g \in G},\{ \os^g \}_{g \in G})$ be a $G$-family of biquandles.
Then $X \times G =\bigsqcup_{x \in X} \{ x \} \times G$ is an MCB with 
\begin{align*}
&(x,g)\us(y,h):=(x\us^hy,h^{-1}gh),
\hspace{10mm}
(x,g)(x,h):=(x,gh),\\
&(x,g)\os(y,h):=(x\os^hy,g)
\end{align*}
for any $x,y \in X$ and $g,h \in G$ \cite{IIKKMOpre,IN17}.
We call it the \emph{associated MCB} 
of  $(X,\{ \us^g \}_{g \in G},\{ \os^g \}_{g \in G})$.

Let $(X,\{ \us^g \}_{g \in G},\{ \os^g \}_{g \in G} )$ be a $G$-family of biquandles.
For any $x,y \in X$ and $g \in G$, 
it follows that
\begin{align*}
(x \us^g y) \us^{g^{-1}}(y \us^g y)
=x \us^e y
=x
\end{align*}
and
\begin{align*}
x \us^{g^{-1}}(y \us^g y) \us^g y
&= \{ x \us^{g^{-1}}(y \us^g y) \} \us^g \{ (y \us^g y) \us^{g^{-1}} (y \us^g y) \}\\
&= x \us^e (y \us^g y)\\
&= x.
\end{align*}
Hence the map $\us^g y : X \to X$, 
which sends $x$ into $x \us^g y$, 
is a bijection 
and $(\us^g y)^{-1} (x)=x \us ^{g^{-1}}(y \us^g y)$.
Similarly, 
the map $\os^g y : X \to X$, 
which sends $x$ into $x \os^g y$, 
is a bijection 
and $(\os^g y)^{-1} (x)=x \os ^{g^{-1}}(y \os^g y)$.
Then we have the following proposition.

\begin{proposition}\label{Gbiqnd to Gqnd}
Let $(X,\{ \us^g \}_{g \in G},\{ \os^g \}_{g \in G} )$ be a $G$-family of biquandles.
Then $(X,\{ *^g \}_{g \in G})$ is a $G$-family of quandles 
by defining $x *^g y =(x \us^g y) \os^{g^{-1}}(y \os^g y)$.
\end{proposition}

\begin{proof}
\begin{itemize}
\item
For any $x \in X$ and $g \in G$, 
\begin{align*}
x*^gx 
=(x \us^g x) \os^{g^{-1}}(x \os^g x)
=(x \os^g x) \os^{g^{-1}}(x \os^g x)
=x \os^e x
=x.
\end{align*}
\item
For any $x,y \in X$, $g,h \in G$, 
\begin{align*}
&x*^{gh}y \os^g y \os^h (y \os^g y)\\
&= x \us^{gh} y \os^{h^{-1}g^{-1}} (y \os^{gh} y) \os^g y \os^h (y \os^g y) \\
&= x \us^{gh} y \os^{h^{-1}} (y \os^{gh} y) \os^{g^{-1}} \{ (y \os^{gh} y) \os^{h^{-1}}(y \os^{gh} y)\} \os^g y \os^h (y \os^g y) \\
&= x \us^{gh} y \os^{h^{-1}} (y \os^{gh} y) \os^{g^{-1}} (y \os^g y) \os^g y \os^h (y \os^g y) \\
&= x \us^{gh} y \os^{h^{-1}} (y \os^{gh} y) \os^h (y \os^g y) \\
&= x \us^{gh} y \os^{h^{-1}} \{ (y \os^g y) \os^h (y \os^g y) \} \os^h (y \os^g y) \\
&= x \us^{gh} y.
\end{align*}
On the other hand, 
\begin{align*}
&(x *^g y) *^h y \os^g y \os^h (y \os^g y)\\
&= \{x *^g y \us^h y \os^{h^{-1}} (y \os^h y) \os^g y\} \os^h (y \os^g y) \\
&= \{x *^g y \us^h y \os^{h^{-1}} (y \os^h y) \os^h y\} \os^{h^{-1}gh} (y \us^h y) \\
&= (x *^g y \us^h y) \os^{h^{-1}gh} (y \us^h y) \\
&= \{ (x \us^g y) \os^{g^{-1}}(y \os^g y) \os^g y\} \us^h (y \os^g y) \\
&= (x \us^g y) \us^h (y \us^g y) \\
&= x \us^{gh} y.
\end{align*}
Therefore we have $x *^{gh} y= (x *^g y)*^h y$.
we can easily check that $x*^e y=x$ for any $x,y \in X$.
\item
For any $x,y,z \in X$, $g,h \in G$ and $\alpha = (y \us^h z) \os^{h^{-1}} (z \os^h z)$, 
\begin{align*}
&(x *^g y)*^h z  \os^{h^{-1}gh} \alpha \os^{h^{-1}g^{-1}hgh} (z \us^{h^{-1}gh} \alpha) \\
&=\{( (x \us^g y) \os^{g^{-1}} (y \os^g y) \us^h z \os^{h^{-1}} (z \os^h z)) \os^{h^{-1}gh} \alpha \} \os^{h^{-1}g^{-1}hgh} (z \us^{h^{-1}gh} \alpha) \\
&=\{( (x \us^g y) \os^{g^{-1}} (y \os^g y) \us^h z \os^{h^{-1}} (z \os^h z)) \os^h z \} \os^{h^{-1}gh} (\alpha \os^h z) \\
&=\{(x \us^g y) \os^{g^{-1}} (y \os^g y) \us^h z\} \os^{h^{-1}gh} (y \us^h z) \\
&=\{(x \us^g y) \os^{g^{-1}} (y \os^g y) \os^g y\} \us^h (z \os^g y) \\
&=(x \us^g y) \us^h (z \os^g y).
\end{align*}
On the other hand, 
\begin{align*}
& (x *^h z)*^{h^{-1}gh} (y *^h z) \os^{h^{-1}gh} \alpha \os^{h^{-1}g^{-1}hgh} (z \us^{h^{-1}gh} \alpha)\\
&= ((x \us^h z) \os^{h^{-1}} (z \os^h z)) *^{h^{-1}gh} ((y \us^h z) \os^{h^{-1}} (z \os^h z)) \os^{h^{-1}gh} \alpha \os^{h^{-1}g^{-1}hgh} (z \us^{h^{-1}gh} \alpha) \\
&= ((x \us^h z) \os^{h^{-1}} (z \os^h z)) \us^{h^{-1}gh} \alpha \os^{h^{-1}g^{-1}h} (\alpha \os^{h^{-1}gh} \alpha )\os^{h^{-1}gh} \alpha \os^{h^{-1}g^{-1}hgh} (z \us^{h^{-1}gh} \alpha) \\
&= \{ ((x \us^h z) \os^{h^{-1}} (z \os^h z)) \us^{h^{-1}gh} \alpha \}  \os^{h^{-1}g^{-1}hgh} ( z \us^{h^{-1}gh} \alpha )\\
&=\{ ((x \us^h z) \os^{h^{-1}} (z \os^h z)) \os^h z \} \us^{h^{-1}gh} \{ ((y \us^h z) \os^{h^{-1}} (z \os^h z)) \os^h z \} \\
&=(x \us^h z) \us^{h^{-1}gh} (y \us^h z)\\
&=(x \us^g y) \us^h (z \os^g y).
\end{align*}
Therefore we have $(x *^g y)*^h z= (x *^h z)*^{h^{-1}gh} (y *^h z)$.
\end{itemize}
\end{proof}

By Proposition \ref{Gbiqnd to Gqnd}, 
for any $G$-family of biquandles $(X,\{ \us^g \}_{g \in G},\{ \os^g \}_{g \in G})$, 
we have a $G$-family of quandles $(X,\{*^g\}_{g \in G})$,
denoted by $\Q_G(X)$, 
by defining $x *^g y =(x \us^g y) \os^{g^{-1}}(y \os^g y)$.
Then $\Q_G$ is a map from the set of $G$-families of biquandles to that of $G$-families of quandles.
In particular, 
let $(X,\{ \us^g \}_{g \in G},\{ \os^g \}_{g \in G})$ be a $G$-family of Alexander biquandles, 
where $X$ is a right $R[G]$-module for some ring $R$ and group $G$ with a homomorphism $\phi : G \to Z(G)$.
Then $\Q_G(X)$ is a $G$-family of Alexander quandles with the action $xg:=xg\phi(g)$ 
since 
for any $x,y \in X$ and $g \in G$,
we have 
$x *^g y =(x \us^g y) \os^{g^{-1}}(y \os^g y)=xg\phi(g) +y(e-g\phi(g))$.

For any $G$-family of biquandles $(X,\{ \us^g \}_{g \in G},\{ \os^g \}_{g \in G})$ and its associated MCB $X \times G$, 
the MCQ $\Q(X \times G)$ coincides with the associated MCQ $\Q_G(X) \times G$ of 
the $G$-family of quandles $\Q_G(X)$ 
with $(x,g)*(y,h)=((x \us^h y) \os^{h^{-1}}(y \os^h y),h^{-1}gh)$.

Let $G$ be a group 
and let $D$ be a diagram of an $S^1$-oriented handlebody-link $H$.
A \emph{$G$-flow} of $D$ is a map $\phi : \mathcal{A}(D) \to G$ satisfying the conditions 
depicted in Figure \ref{flow} 
at each crossing and vertex.
In this paper, 
to avoid confusion, 
we often represent an element of $G$ with an underline.
We denote by $(D,\phi)$, 
which is called a \emph{$G$-flowed diagram} of $H$, 
a diagram $D$ given a $G$-flow $\phi$, 
and by $\flow(D;G)$ the set of all $G$-flows of $D$.
We can identify a $G$-flow $\phi$ 
with a homomorphism from the fundamental group $\pi_1(S^3-H)$ to $G$.

\begin{figure}[htb]
\begin{center}
\includegraphics[width=77mm]{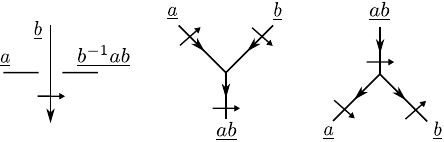}
\end{center}
\caption{A $G$-flow of $D$.}\label{flow}
\end{figure}

Let $G$ be a group 
and let $D$ be a diagram of an $S^1$-oriented handlebody-link $H$.
Let $D'$ be a diagram obtained by applying one of Reidemeister moves 
to the diagram $D$ once.
For any $G$-flow $\phi$ of $D$, 
there is a unique $G$-flow $\phi'$ of $D'$ 
which coincides with $\phi$ 
except near the point where the move is applied.
Therefore 
the cardinality of the set of $G$-flows of $D$, 
denoted by $\# \flow(D;G)$, 
is an invariant of $H$.
We call the $G$-flow $\phi'$ 
the \emph{associated $G$-flow} of $\phi$ 
and the $G$-flowed diagram $(D',\phi')$ 
the \emph{associated $G$-flowed diagram} of $(D,\phi)$.

Let $X$ be a $G$-family of quandles (resp. biquandles) 
and let $(D,\phi)$ be a $G$-flowed diagram of an $S^1$-oriented handlebody-link.
An \emph{$X$-coloring} of $(D,\phi)$ is a map 
$C : \mathcal{A}(D,\phi) \to X$ (resp. $\mathcal{SA}(D,\phi) \to X$ ) satisfying 
the conditions 
depicted in Figure \ref{G-family_of_qnd_col.} (resp. Figure \ref{G-family_of_biqnd_col.}) 
at each crossing and vertex.
We denote by $\col_X(D,\phi)$ 
the set of all $X$-colorings of $(D,\phi)$.
We note that 
when $X$ is a $G$-family of Alexander (bi)quandles 
that is a right $R[G]$-module for some ring $R$, 
the set $\col_X(D,\phi)$ is a right $R$-module 
with the action $(C \cdot r) (\alpha) := C(\alpha)r$ 
and the addition $(C + C')(\alpha):= C(\alpha)+C'(\alpha)$ 
for any $C, C' \in \col_X(D,\phi)$, $\alpha \in \mathcal{A}(D,\phi)$ (or $\alpha \in \mathcal{SA}(D,\phi)$) and $r \in R$.

\begin{figure}[htb]
\begin{center}
\includegraphics[width=85mm]{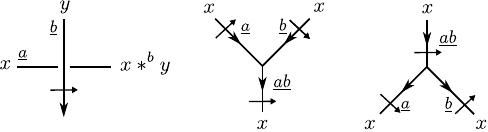}
\end{center}
\caption{A $G$-family of quandles coloring of $(D,\phi)$.}\label{G-family_of_qnd_col.}
\end{figure}

\begin{figure}[htb]
\begin{center}
\includegraphics[width=125mm]{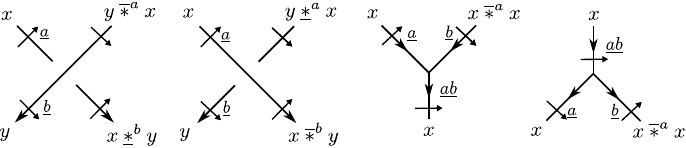}
\end{center}
\caption{A $G$-family of biquandles coloring of $(D,\phi)$.}\label{G-family_of_biqnd_col.}
\end{figure}

\begin{proposition}[\cite{IIJO13,IN17}]
Let $X$ be a $G$-family of (bi)quandles  
and let $(D,\phi)$ be a $G$-flowed diagram of an $S^1$-oriented handlebody-link.
Let $(D',\phi')$ be the associated $G$-flowed diagram of $(D,\phi)$.
For any $X$-coloring $C$ of $(D,\phi)$, 
there is a unique $X$-coloring $C'$ of $(D',\phi')$ 
which coincides with $C$ 
except near the point where the move is applied. 
\end{proposition}

By this proposition, 
we have $\# \col_X(D,\phi) = \# \col_X(D',\phi')$.

Let $X$ be a $G$-family of quandles (resp. biquandles), 
$X \times G$ the associated MCQ (resp. MCB) of $X$ 
and let $\pr_G$ and $\pr_X$ be the natural projections from $X \times G$ to $G$ and from $X \times G$ to $X$ respectively.
For any $\phi \in \flow (D;G)$, 
we define $\col_{X \times G}^\phi (D):=\{ C \in \col_{X \times G}(D) ~|~ \pr_G \circ C=\phi \}$, 
where for any $\alpha \in \mathcal{SA}(D)$ and $\widetilde{\alpha} \in \mathcal{A}(D)$ satisfying $\alpha \subset \widetilde{\alpha}$, 
we put $\phi(\alpha):=\phi(\widetilde{\alpha})$ 
when $X$ is a $G$-family of biquandles.
Then we can identify $\col_{X \times G}^\phi (D)$ with $\col_X(D,\phi)$, 
that is, 
for any $C \in \col_{X \times G}^\phi (D)$, 
the map $\pr_G \circ C$ corresponds to the $G$-flow $\phi$ of $D$, 
and the map $\pr_X \circ C$ corresponds to the $X$-coloring of $(D,\phi)$.
Therefore $\col_{X \times G}^\phi (D)$ is also a right $R$-module 
in the same way as $\col_X(D,\phi)$.
Then we obtain the following corollary by Theorem \ref{1:1 correspondence}.

\begin{corollary}\label{isomorphic}
Let  $X$ be a $G$-family of biquandles 
and let $(D,\phi)$ be a $G$-flowed diagram of an $S^1$-oriented handlebody-link.
Then there is a one-to-one correspondence between $\col_X(D,\phi)$ and $\col_{\Q_G(X)}(D,\phi)$.
In particular, when $X$ is a $G$-family of Alexander biquandles, 
$\col_X(D,\phi)$ is isomorphic to $\col_{\Q_G(X)}(D,\phi)$ as right $R$-modules.
\end{corollary}

\begin{proof}
We remind that 
we can identify $\col_{X \times G}^\phi(D)$ with $\col_X(D,\phi)$ and $\col_{\Q_G(X) \times G}^\phi(D)$ with  $\col_{\Q_G(X)}(D,\phi)$, 
and we note that 
$\col_{X \times G}(D)=\bigsqcup_{\phi' \in \flow(D;G)}\col_{X \times G}^{\phi'}(D)$ 
and 
$\col_{\Q(X \times G)}(D)=\col_{\Q_G(X) \times G}(D)=\bigsqcup_{\phi' \in \flow(D;G)}\col_{\Q_G(X) \times G}^{\phi'}(D)$.
By the proof of Theorem \ref{1:1 correspondence}, 
the map $\Psi^{X \times G}:= \psi_3^{X \times G} \circ \psi_2^{X \times G} \circ \psi_1^{X \times G}$ is a bijective map 
from $\col_{\Q_G(X) \times G}(D)$ to $\col_{X \times G}(D)$, 
and 
$\Psi^{X \times G}(\col_{\Q_G(X) \times G}^{\phi'}(D)) \subset \col_{X \times G}^{\phi'}(D)$ 
for any $\phi' \in \flow(D;G)$ 
(see Figures \ref{psi1} and \ref{deformation}).
Hence $\Psi^{X \times G} |_{\col_{\Q_G(X) \times G}^\phi(D)}$ is a bijective map 
from $\col_{\Q_G(X) \times G}^\phi(D)$ to $\col_{X \times G}^\phi(D)$.
Next, suppose that $X$ is a $G$-family of Alexander biquandles.
Then $\psi_1^{X \times G}$ and $\psi_3^{X \times G}$ preserve module structures clearly.
Furthermore 
$\psi_2^{X \times G}$ also preserves module structures 
since 
in Lemma \ref{braid coloring}, 
each $y_i$ can be represented by using each $x_i$ and the operations $\us$ and $\os$.
Therefore $\Psi^{X \times G} |_{\col_{\Q_G(X) \times G}^\phi(D)}$ is an isomorphism of right $R$-modules.
\end{proof}

Finally, we see an example.
Let $H_n$ be the handlebody-knot represented by the $\mathbb{Z}_8$-flowed diagram $(D_n,\phi_n)$ 
depicted in Figure \ref{H_n} for any $n \in \mathbb{Z}_{>0}$.
Let $s=t+1 \in \mathbb{Z}_3[t^{\pm 1}]$ 
and let $f(t)=t^2+t+2 \in \mathbb{Z}_3[t^{\pm 1}]$, which is an irreducible polynomial.
Then $X:=\mathbb{Z}_3[t^{\pm 1}]/(f(t))$ is 
a $\mathbb{Z}_8$-family of Alexander biquandles and a field.
By \cite[Example 7.3]{Murao17}, 
it follows that $\dim \col_X(D_n,\phi_n)=n$ as vector spaces over $X$, 
and 
the assignment of elements $x_1,\ldots,x_n$ of $X$ to each semi-arc of $(D_n,\phi_n)$ 
as shown in Figure \ref{H_n} 
corresponds to a basis of $\col_X(D_n,\phi_n)$.
By Proposition \ref{Gbiqnd to Gqnd}, 
$\Q_G(X)$ is a $\mathbb{Z}_8$-family of Alexander quandles 
with $x*^iy=s^{-i}t^ix+(1-s^{-i}t^i)y=t^{2i}x+(1-t^{2i})y$ for any $i \in \mathbb{Z}_8$.
By Corollary \ref{isomorphic}, 
we have $\dim \col_{\Q_G(X)}(D_n,\phi_n)=n$ as vector spaces over $X$, 
and 
the assignment of elements $x_1,\ldots,x_n$ of $X$ to each arc of $(D_n,\phi_n)$ 
as shown in Figure \ref{H_n} 
corresponds to a basis of $\col_{\Q_G(X)}(D_n,\phi_n)$.

\begin{figure}[htb]
\begin{center}
\includegraphics[width=70mm]{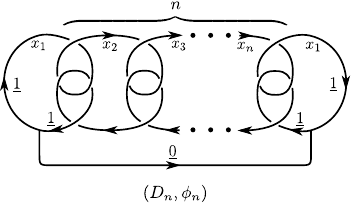}
\end{center}
\caption{A $\mathbb{Z}_8$-flowed diagram $(D_n,\phi_n)$ of $H_n$.}\label{H_n}
\end{figure}

\section*{acknowledgement}
The author would like to express his best gratitude to Katsumi Ishikawa and Kokoro Tanaka 
for their helpful advice and valuable discussions.
He also would like to express his thanks to Atsushi Ishii 
for his  beneficial comments and suggestions.


\end{document}